\documentclass{amsart}

\usepackage[all]{xy}
\usepackage{hyperref}
\hypersetup{
    colorlinks=true,       
    linkcolor=blue,       
    citecolor=blue,       
    filecolor=blue,      
    urlcolor=blue        
}
\usepackage[lite]{amsrefs}
\usepackage{amssymb}
\usepackage{todonotes}
\usepackage{stmaryrd}
\usepackage{tikz-cd}

\newcommand{\loc}{\operatorname{loc}}

\newcommand{\codim}{\operatorname{codim}}

\newcommand{\supp}{\operatorname{supp}}

\newcommand{\Spec}{\operatorname{Spec}}
\newcommand{\Chow}{\operatorname{Chow}}

\newcommand{\Hom}{\operatorname{Hom}}

\newcommand{\relint}{\operatorname{relint}}

\newcommand{\Pol}{\operatorname{Pol}}
\newcommand{\CaDiv}{\operatorname{CaDiv}}

\newcommand{\triv}{\operatorname{triv}}

\newcommand{\pp}{\mathbb{P}}

\newcommand{\qq}{\mathbb{Q}}
\newcommand{\zz}{\mathbb{Z}}

\newcommand{\cc}{\mathbb{C}}

\renewcommand{\div}{\operatorname{div}}

\newtheorem{introthm}{Theorem}
\newtheorem{introcor}{Corollary}

\newtheorem{theorem}{Theorem}[section]
\newtheorem{lemma}[theorem]{Lemma}

\newtheorem{corollary}[theorem]{Corollary}

\theoremstyle{definition}

\newtheorem{notation}[theorem]{Notation}
\newtheorem{definition}[theorem]{Definition}
\newtheorem{example}[theorem]{Example}

\newtheorem{construction}[theorem]{Construction}

\newtheorem{remark}[theorem]{Remark}

\theoremstyle{remark}

\numberwithin{equation}{section}

\begin{document}

\title[Fundamental group of log terminal $\mathbb{T}$-varieties]{Fundamental group of log terminal $\mathbb{T}$-varieties}
\author[A.~Laface]{Antonio Laface}
\address{
Departamento de Matem\'atica,
Universidad de Concepci\'on,
Casilla 160-C,
Concepci\'on, Chile}
\email{alaface@udec.cl}

\author[A. ~Liendo]{Alvaro Liendo}
\address{
Instituto de Matem\'atica y F\'isica,
Universidad de Talca,
Casilla 721,
Talca, Chile}
\email{aliendo@inst-mat.utalca.cl}

\author[J.~Moraga]{Joaqu\'in Moraga}
\address{
Department of Mathematics, University of Utah, 155 S 1400 E, 
Salt Lake City, UT 84112}
\email{moraga@math.utah.edu}


\thanks{The first author was partially supported by Proyecto FONDECYT
  Regular N. 1150732 and Projecto Anillo ACT 1415 PIA Conicyt. The
  second author was partially supported by Proyecto FONDECYT Regular
  N. 1160864.}

\maketitle

\medskip 

\begin{abstract}
In this article, we introduce an approach to study the fundamental group of a log terminal $\mathbb{T}$-variety.
As applications, we prove the simply connectedness of the spectrum of the Cox ring of a complex Fano variety,
we compute the fundamental group of a rational log terminal $\mathbb{T}$-varieties of complexity one, 
and we study the local fundamental group of log terminal $\mathbb{T}$-singularities with a good torus action and trivial GIT decomposition.
\end{abstract}

\setcounter{tocdepth}{1}
\tableofcontents

\section*{Introduction}

We study the fundamental group of normal complex
algebraic varieties endowed with an effective action of an algebraic
torus $\mathbb{T}:=(\cc^*)^k$, these varieties are known as
$\mathbb{T}$-varieties.  The complexity of a $\mathbb{T}$-variety $X$
is defined to be $\dim(X)-k$.  The $\mathbb{T}$-varieties of
complexity zero are the classic toric varieties that can be described
in terms of fans of polyhedral cones (see, e.g., ~\cite{CLS, Dem, KKMS}).  
In ~\cite{KKMS},~\cite{Tim08} and ~\cite{HF} there are
generalizations of such description to the case of
$\mathbb{T}$-varieties of complexity one.  Finally in ~\cite{AH05,AHH}
the authors introduce the language of {\em polyhedral divisors} and
{\em divisorial fans} to extend the theory of toric varieties to
$\mathbb{T}$-varieties of arbitrary complexity.

In this paper, we are interested in the fundamental group
of the underlying topological space of a complex $\mathbb{T}$-variety.
The topology of toric varieties has been well studied (see,
e.g., \cite{CLS,DJ,Franz,FS,Fu}). 
In the toric case, the fundamental group can be computed 
in terms of the defining fan of the toric variety (see ~\cite[Theorem 12.1.10]{CLS}).
More precisely, the fundamental group of an affine toric variety is a free finitely generated abelian group,
and the fundamentall group of a toric variety is a finitely generated abelian group.
In particular, a complex toric variety with a fixed point for the torus action is simply connected. 
In higher complexity, a $\mathbb{T}$-variety $X$ is determined by a divisorial fan
$\mathcal{S}$, which is a geometric and combinatorial object which depends on certain divisors
on the Chow quotient of $X$ (see Definition~\ref{Chow}) and polyhedra
associated to such divisors in a fixed $\qq$-vector space.
In the case that $X$ is affine with a good $\mathbb{T}$-action (see Definition~\ref{good}) 
we obtain the following result which generalize the toric case:

\begin{introthm}\label{Thm1}
Let $X$ be a complex affine log terminal variety with 
a good $\mathbb{T}$-action and denote by $Y$
a resolution of singularities of its Chow quotient.
Then the pushforward $\pi_*\colon \pi_1(X)\to\pi_1(Y)$
is an isomorphism.
\end{introthm}

First, let us point that the group $\pi_1(Y)$ is independent of the
chosen resolution of singularities (see Remark~\ref{ind}) and that the
log terminal condition in Theorem~\ref{Thm1} cannot be weakened:
For instance, a cone over an elliptic curve $C$ has a log canonical
singularity at the vertex and trivial fundamental group, being
contractible, but its Chow quotient is $C$ which has non-trivial
fundamental group (see Remark~\ref{lc}). As a consequence, 
we obtain the following (see~\cite{ADHL} for a definition
of the Cox ring):

\begin{introcor}\label{Cor1}
Let $X$ be a complex Fano variety, and let $\overline{X}$ be the underlying topological space
of the spectrum of the Cox ring of $X$. Then $\overline{X}$ is simply connected.
\end{introcor}

Then, we focus on the case of rational $\mathbb{T}$-varieties $X$ with a one-dimensional
Chow quotient, the so called $\mathbb{T}$-varieties of {\em complexity one}.
In Theorem~\ref{comp1fund}, we give an explicit description of the fundamental group in terms of
the defining divisorial fan $\mathcal{S}$ of $X$.
For instance, we can use this description to characterize the fundamental group 
of an algebraic $\cc^*$-bundle over $\pp^1$.

\begin{introcor}\label{Cor2}
Let $X\rightarrow \pp^1$ be a $\cc^*$-bundle which is trivial outside $\{p_1,\dots, p_{r}\}$.
Then, we have that
\[
 \pi_1(X) 
 \simeq 
 \langle b_1,\dots, b_{r}, t \mid
 b_1\cdots b_r,\,  [b_i,t],\, t^{e_i}b_i^{m_i}
 \text{ for }1\leq i\leq r \rangle,
\]
where $\Delta_{p_i}=\{ e_i/m_i\}$ is the 
polyhedral coefficient at $p_i$.
\end{introcor}

Finally, we study the local fundamental group of rational log terminal $\mathbb{T}$-varieties with a good torus action.
It is conjectured that the local fundamental group of a log terminal singularity is finite (see, e.g.~\cite{Koll}).
Moreover, this conjecture is known for the algebraic fundamental group~\cite{Xu}, 
and for toric singularities~\cite[Theorem 12.1.10]{CLS}. We extend this latter result to the following context.

\begin{introthm}\label{Thm2}
Let $X$ be a log terminal $\mathbb{T}$-variety with a good torus action and $x\in X$ the vertex.
Assume that the GIT fan of $X$ 
has a unique maximal chamber
and each fiber of $\pi \colon X\dashrightarrow Y$
over a codimension one point, contains a smooth point.
Then, the local fundamental group at $x\in X$ is finite.
\end{introthm}

The paper is organized as follows: 
In Section~\ref{sec:background},
we introduce the combinatorial description of $\mathbb{T}$-varieties via divisorial fans.
In Section~\ref{sec:genapp}, we explain the general approach to compute
$\pi_1(X(\mathcal{S}))$ by using the information of the divisorial fan $\mathcal{S}$.
Section~\ref{sec:app} is devoted to the applications:
In subsection~\ref{subsec:fixed} we prove Theorem~\ref{Thm1} and Corollary~\ref{Cor1},
in subsection~\ref{subsec:comp1} we describe the fundamental group of a rational log terminal $\mathbb{T}$-variety of complexity one,
and in subsection~\ref{subsec:finite} we prove Theorem~\ref{Thm2}.
Finally, in Section~\ref{sec:examples} we give explicit computations in the case of Du Val singularities.

\subsection*{Acknowledgements}
The authors would like to thank J\'anos Koll\'ar and Chenyang Xu for
pointing out a gap in a early version.  The authors would also like to
thank Hendrik S{\"u}\ss\ for providing interesting examples.

\section{Basic Setup}
\label{sec:background}

In this section, we introduce the description
of $\mathbb{T}$-varieties in terms of divisorial fans
due to Altmann, Hausen and S\"u\ss~\cite{AH05,AHH},
see \cite{AIPSV} for a survey on known results about 
the geometry of $\mathbb{T}$-varieties.
The point of view here is to start
with a variety $Y$ together with
a combinatorial data called
a {\em divisorial fan} on $Y$
and construct a $\mathbb T$-variety $X$
whose normalized Chow quotient is $Y$.
We start recalling the definition of Chow quotient:

\begin{definition}\label{good}
Let $X$ be a normal affine variety. 
A {\em good $\mathbb{T}$-action} on $X$ 
is an effective $\mathbb{T}$-action on $X$
such that there exists a closed point $x\in X$ 
which is in the closure of any $\mathbb{T}$-orbit.
We shall call $x$ the {\em vertex point} of $X$.
\end{definition}

\begin{definition}\label{Chow}
Let $X$ be a $\mathbb{T}$-variety embedded in a projective space $\pp^N$, 
then there exists an open set of $X$ on which 
all the orbits have dimension $k$ and degree $d$, 
the {\em Chow quotient} of $X$ is the closure of the set
of points corresponding to such orbits in $\Chow_{k,d}(\pp^N)$,
the Chow variety parametrizing cycles of dimension $k$
and degree $d$ on $\pp^N$.
The isomorphism class of the Chow quotient is independent
from the chosen embedding.
The normalization of the Chow quotient of $X$ will be called
the {\em normalized Chow quotient}.
\end{definition}

Now we turn to introduce the language of 
{\em polyhedral divisors} and {\em divisorial fans}:
Given $N$ a finitely generated free abelian group
of rank $k$ we will denote by $M:=\Hom(N,\zz)$ its dual
and by $N_\qq :=N\otimes_\zz \qq$ 
and $M_\qq:=M\otimes_\zz \qq$ the associated $\qq$-vector spaces.
We will denote by $\mathbb{T}_N:=\Spec \cc[M] \simeq (\cc^*)^k$
the torus of $N$.
For every convex polyhedron 
$\Delta \subsetneq N_\qq$ one defines its {\em tail cone} as
\[
\sigma(\Delta) := \{ v \in N_\qq \mid v+ \Delta \subset \Delta \}.
\]
A polyhedron $\Delta$ with tail cone $\sigma$ will be called
a {\em $\sigma$-polyhedron}.
The set of $\sigma$-polyhedra is denoted by $\Pol(\sigma)$.
Observe that $\Pol(\sigma)$ endowed with the Minkowski is a semigroup.

We adopt the notation $\CaDiv(Y)_\qq$ for the monoid
of $\qq$-Cartier $\qq$-divisors of a normal variety $Y$.
A {\em polyhedral divisor} on $(Y,N)$ is a finite formal sum of the form
\[
\mathcal{D}:= \sum_{D} \Delta_D \otimes D \in \Pol(\sigma) \otimes_{\zz_{\geq 0}}\CaDiv(Y)_\qq
\]
where the sum is taken over a finite set of prime $\qq$-Cartier
$\qq$-divisors, and $\Delta_D$ are convex $\sigma$-polyhedra in
$N_\qq$.  The common tail cone of the polyhedra $\Delta_D$ is called
the {\em tail cone of $\mathcal{D}$} and is denoted by
$\sigma(\mathcal{D})$.  
The $\sigma(\mathcal{D})$-polyhedra $\Delta_D$ associated to the divisor $D$ wil be called the 
{\em polyhedral coefficient} of $D$.
We will also consider the enlarged monoid
$\Pol^+(\sigma):= \Pol(\sigma) \cup \{ \emptyset \}$ with addition
rule $\emptyset +\Delta := \emptyset$ for every
$\Delta \in \Pol^+(\sigma)$.  The {\em locus} of a polyhedral divisor
$\mathcal{D}$ is defined as
\[
\loc(\mathcal{D}):= Y - \bigcup_{\Delta_D = \emptyset }  D
\]
and we say that $\mathcal{D}$ has {\em complete locus}
whenever $\loc(\mathcal{D})=Y$ meaning that there is no 
$\qq$-divisor $D\subset Y$ with coefficient $\emptyset$.
The {\em support} of $\mathcal{D}$ is 
\[
\supp(\mathcal{D}):= \loc(\mathcal{D}) \cap \bigcup_{\Delta_D \neq \sigma} D
\]
and the {\em trivial locus} of $\mathcal{D}$
is the complement of the support of $\mathcal{D}$
in the locus of $\mathcal{D}$,
and is denoted by $\triv(\mathcal{D})$.

Let $\mathcal{D}$ be a polyhedral divisor on $(Y,N)$
with tail cone $\sigma$.
We have a natural homomorphism of monoids
\[
 \mathcal{D}\colon \sigma^\vee \rightarrow \CaDiv(Y)
\qquad u\mapsto
 \mathcal{D}(u)
 :=\sum \min_{v\in \Delta_D} \langle u,v \rangle D.
\]
Observe that $\mathcal{D}(u)+\mathcal{D}(u') \leq
\mathcal{D}(u+u')$ holds for every $u,u'\in \sigma^\vee$
and that the support of any divisor $\mathcal{D}(u)$
is contained in the support of $\mathcal{D}$.

\begin{definition}\label{p-div}
A $\qq$-divisor $D$ is said to be {\em semiample}
if it admits a base point free multiple
and is said to be {\em big} if some multiple admits
a section with affine complement.
A polyhedral divisor $\mathcal{D}$ is said to be a {\em proper polyhedral divisor}
if $\mathcal{D}(u)$ is semiample for every $u\in \sigma(\mathcal{D})^\vee$
and $\mathcal{D}(u)$ is big for $u\in \relint(\sigma(\mathcal{D})^\vee)$.
In order to shorten notation, we will say that a proper polyhedral divisor $\mathcal{D}$
is a {\em pp-divisor}.
\end{definition}

We recall the relation between affine $\mathbb{T}$-varieties and pp-divisors.
Given a pp-divisor $\mathcal{D}$ on $(Y,N)$ 
one defines the $M$-graded 
$\mathcal{O}_{\loc(\mathcal{D})}$-algebra
\[
 \mathcal{A}(\mathcal{D})
 :=
 \bigoplus_{u\in \sigma^\vee \cap M} \mathcal{O}_{\loc(\mathcal{D})}(\mathcal{D}(u)).
\]
The $M$-grading induces an effective 
$\mathbb{T}_N$ action on both the relative
spectrum $\widetilde X(\mathcal D)$ and the
spectrum of global sections $X(\mathcal D)$
of the above sheaf of algebras. The inclusion
$\mathcal{O}_{\loc(\mathcal{D})} 
\to \mathcal{A}(\mathcal{D})$ induces a good 
quotient morphism $\pi \colon \widetilde X(\mathcal{D})
\to \loc(\mathcal{D})$. 
The construction is summarized in the 
following diagram
\[
 \xymatrix{
 \widetilde X(\mathcal D)\ar[d]_-\pi\ar[r]^\phi & X(\mathcal D)\\
 \loc(\mathcal{D})
 }
\]
where the natural morphism $r$
can be proved to be a $\mathbb{T}_N$-equivariant 
birational contraction.
The main result in ~\cite{AH05} states that every normal affine
$\mathbb{T}$-variety arises in this way.

In what follows we will describe the gluing process of 
affine $\mathbb{T}$-varieties in terms of pp-divisors.
Given two pp-divisors $\mathcal{D}$ and $\mathcal{D}'$ on $(Y,N)$ we
write $\mathcal{D}'\subseteq \mathcal{D}$ if 
$\Delta'_{D} \subseteq \Delta_D$ for every
$\qq$-Cartier $\qq$-divisor $D\subset Y$.
If $\mathcal{D}' \subset \mathcal{D}$ 
then we have an induced morphism
$X(\mathcal{D}')\rightarrow X(\mathcal{D})$
and we say that $\mathcal{D'}$
is a face of $\mathcal{D}$ if the 
induced morphism is an embedding, 
and we denote this relation by
$\mathcal{D}' \leq \mathcal{D}$.
If $\mathcal{D}'$ is a face of $\mathcal{D}$
then in particular we have that 
$\sigma(\mathcal{D}') \leq \sigma(\mathcal{D})$.
The intersection of two pp-divisors 
$\mathcal{D}$ and $\mathcal{D}'$ 
is defined to be the polyhedral divisor
$\mathcal{D}\cap \mathcal{D}':=
\sum_D (\Delta_D \cap \Delta'_D)\otimes D$.

\begin{definition}\label{df}
A set $\mathcal{S}$ of pp-divisors is said to be 
a {\em divisorial fan} if it holds the following conditions:
\begin{itemize}
\item $\mathcal{S}$ is finite,
\item $\mathcal{S}$ is closed under taking intersection,
\item the intersection of any two pp-divisors of $\mathcal{S}$
is a face of both.
\end{itemize}
\end{definition}

Gluing the affine $\mathbb{T}$-varieties $X(\mathcal{D})$ and
$X(\mathcal{D}')$ along the affine subvarieties
$X(\mathcal{D}\cap \mathcal{D}')$ for every $\mathcal{D}$ and
$\mathcal{D}'$ in $\mathcal{S}$, we obtain a $\mathbb{T}$-variety
$X(\mathcal{S})$ (see ~\cite[Section 4.4]{AIPSV} for details).
Analogously for the $\mathbb{T}$-varieties $\widetilde X(\mathcal{D})$
with $\mathcal{D}\in \mathcal{S}$ we obtain a $\mathbb{T}$-variety
$\widetilde X(\mathcal{S})$ with two equivariant morphisms
\[
 \xymatrix{
 \widetilde X(\mathcal{S})\ar[d]_-\pi\ar[r]^\phi & X(\mathcal S)\\
 Y
 }
\]
where $r$ is an equivariant birational 
contraction and $\pi$ is a quotient for
the torus action.
The main result in ~\cite{AHH} states that all normal
$\mathbb{T}$-varieties arise this way.

The set $\{ \sigma(\mathcal{D}) \mid \mathcal{D}\in \mathcal{S}\}$
is the {\em tail fan } $\Sigma(\mathcal{S})$ of $\mathcal{S}$.
Observe that the tail fan $\Sigma(\mathcal{S})$ is indeed a fan.
We define the {\em locus} of $\mathcal{S}$ to be the set
\[
\loc(\mathcal{S}):=\bigcup_{\mathcal{D}\in \mathcal{S}} 
\loc(\mathcal{D}),
\]
the {\em support} of $\mathcal{S}$ to be the set
\[
\supp(\mathcal{S}):=\bigcup_{\mathcal{D}\in \mathcal{S}}
\supp(\mathcal{D}),
\]
and the {\em trivial locus} of $\mathcal{S}$ to be
\[
\triv(\mathcal{S}):= \bigcap_{\mathcal{D}\in \mathcal{S}}
\triv(\mathcal{D}).
\]
When $\mathcal{S}$ is the only divisorial fan on $Y$
we may denote $\triv(\mathcal{S})$ by $Y_{\triv}$.

\section{General approach to compute the fundamental group}
\label{sec:genapp}

In this section, we explain an approach to compute the fundamental group 
of a complex log terminal $\mathbb{T}$-variety $X(\mathcal{S})$ using its defining divisorial fan $\mathcal{S}$.
First, we recall some definitions and results from ~\cite{LS}.

\begin{definition}
Consider a pp-divisor $\mathcal{D}$ on $(Y,N)$ with tail cone $\sigma$, 
and a morphism $\psi \colon Y' \rightarrow Y$, 
such that no irreducible component of $\supp(\mathcal{D})$ contains $\psi(Y')$.
The {\em polyhedral pull back} is
\[
\psi^*(\mathcal{D})= \sum_{D} \Delta_D \otimes \psi^*(D) 
\in
\Pol(\sigma) \otimes_{\zz_{\geq 0}}\CaDiv(Y')_\qq.
\]
Observe that $\psi^*(\mathcal{D})$ is a polyhedral divisor which may not be a pp-divisor,
meaning that $\psi^*(\mathcal{D})$ may not be a proper polyhedral divisor.
Given a divisorial fan $\mathcal{S}=\{ \mathcal{D}_i \mid i \in I\}$ on $(Y,N)$,
its pull back is $\psi^*(\mathcal{S}) =\{ \psi^*(\mathcal{D}_i) \mid i \in I\}$.
\end{definition}

\begin{lemma}\label{resolution}
Let $\mathcal{S}$ be a divisorial fan on $(Y,N)$.
If $\psi \colon Y'\rightarrow Y$ is a projective birational morphism, then
$\psi^*(\mathcal{S})$ is a divisorial fan.
Moreover $\psi$ induces a $\mathbb{T}_N$-equivariant 
isomorphism $X(\psi^*(\mathcal{S}))\simeq X(\mathcal{S})$.
\end{lemma}

\begin{proof}
Without loss of generality we assume that all the 
pp-divisors of $\mathcal{S}$ have complete locus:
in case it is not complete, replace $Y$ by 
${\rm Loc}(\mathcal{D})$ in the whole proof.
Given a pp-divisor $\mathcal{D}\in \mathcal{S}$ 
and an element $u\in \sigma(\mathcal{D})^\vee$
we have
\[
\psi^*(\mathcal{D})(u) = \sum_D \min_{v \in \Delta_D} \langle u, v \rangle \psi^*(D)  = 
\psi^*\left( \sum_D \min_{v \in \Delta_D} \langle u, v \rangle D \right) = \psi^*(\mathcal{D}(u)).
\]
Therefore, $\psi^*(\mathcal{D})(u)$ is semiample.
Moreover, if $u\in \relint(\sigma(\mathcal{D})^\vee)$, 
we know that $\mathcal{D}(u)$ is big, 
and the pull-back of a big divisor with respect 
to a projective birational map is again big,
so we conclude that $\psi^*(\mathcal{D}(u))$ is a big $\qq$-divisor.  
Thus, $\psi^*(\mathcal{D})$ is a pp-divisor.

Now we prove that $\psi$ induces a $\mathbb{T}$-equivariant 
isomorphism $X(\psi^*(\mathcal{D}))\simeq X(\mathcal{D})$.
First of all observe that if $D$ 
is a $\mathbb Q$-Weil divisor on $Y$,
and $f$ is a rational function then 
$\div(f) + D\geq 0$ if and only 
if $\div(f) + \lfloor D\rfloor\geq 0$.
Now let $u\in \sigma(\mathcal{D})^\vee$ and
let $m$ be the Cartier index of $\mathcal{D}(u)$.
Observe that we have
\begin{align}
\nonumber &f \in \Gamma(Y, \mathcal{O}_Y(\mathcal{D}(u)))\Leftrightarrow \\
\nonumber &f^m \in \Gamma(Y, \mathcal{O}_Y(m\mathcal{D}(u))) \Leftrightarrow \\
\nonumber &f^m \in \Gamma(Y',\mathcal{O}_{Y'}(m\psi^*(\mathcal{D}(u)))) \Leftrightarrow\\
\nonumber &f \in \Gamma(Y',\mathcal{O}_{Y'}(\psi^*\mathcal{D}(u))),
\end{align}
where the first and last equivalences are 
by the previous observation, while the 
second equivalence follows from the 
fact that $\psi$ is birational
with connected fibers and the projection formula.
So, there is a natural isomorphism of $M$-graded $\mathcal{O}_Y$-algebras
$\mathcal{A}(\psi^*(\mathcal{D})) \simeq \mathcal{A}(\mathcal{D})$, concluding the claim. 
The isomorphism $X(\psi^*(\mathcal{S})) \simeq X(\mathcal{S})$ follows from gluing the above 
$\mathbb{T}_N$-equviariant affine isomorphisms.
\end{proof}

\begin{remark}
The $\mathbb{T}$-varieties $\widetilde{X}(\psi^*(\mathcal{S}))$ and 
$\widetilde{X}(\mathcal{S})$ are isomorphic if and only if the projective birational map
$\psi \colon Y' \rightarrow Y$ is the identity (see, e.g.,~\cite[Section 2]{LS}).
\end{remark}

\begin{definition}
A pp-divisor $\mathcal{D}$ on $(Y,N)$ is {\em simple normal crossing} if $Y$ is smooth and
the support of $\mathcal{D}$ is a divisor with simple normal crossing support. 
Analogously, a divisorial fan $\mathcal{S}$ on $(Y,N)$ is {\em simple normal crossing}
if all its pp-divisors $\mathcal{D}\in \mathcal{S}$ are simple normal crossing.
\end{definition}

\begin{definition}
An algebraic variety $X$ is {\em toroidal} if for each point $x\in X$ there
exists a formal neighborhood of $x$ on $X$ which is isomorphic 
to a formal neighborhood of a point in an affine toric variety.
\end{definition}

The following lemma is proved in ~\cite[Proposition 2.6]{LS}.

\begin{lemma}\label{toroidal}
Let $\mathcal{S}$ be a divisorial fan on $(Y,N)$. If the divisorial fan $\mathcal{S}$ is 
simple normal crossing, then the $\mathbb{T}$-variety $\widetilde{X}(S)$ is toroidal.
\end{lemma}

\begin{definition}\label{strata}
Given a simple normal crossing pp-divisor $\mathcal{D}$ on $(Y,N)$ we can write $\mathcal{D}=\sum_{D} \Delta_D \otimes D$.
Given the prime divisors $D_1, \dots, D_r$ such that $\Delta_{D_i} \neq \sigma$,
we define the {\em strata} of the prime divisors $D_1,\dots, D_r$, to be the locally closed set
\[
Z_{D_1,\dots,D_r} = D_1\cap\dots\cap D_r - \bigcup_{\Delta_D\neq \sigma} D_1\cap\dots\cap D_r\cap D.
\]
Observe that the trivial open set of $\mathcal{D}$ is the strata of the empty set of prime divisors.
The stratas of $\mathcal{D}$ give a natural stratification of $Y$. 
We define a {\em strata} $Z$ of the divisorial fan $\mathcal{S}$ to be a finite intersection of strata of the pp-divisors $\mathcal{D}\in\mathcal{S}$.
Clearly, the divisorial fan $\mathcal{S}$ define a natural stratification of $Y$.			
\end{definition}

\begin{remark}\label{ngbhd}
In this remark we will describe a formal neighborhood of the preimage 
on $\widetilde{X}(\mathcal{S})$ of a strata $Z$ on $Y$ as defined in~\ref{strata}.
First, we consider the case of a single proper polyhedral divisor $\mathcal{D}$ on $(Y,N)$.
By virtue of Lemma~\ref{resolution}, we may assume that the projective variety $Y$ is smooth
and the polyhedral divisor $\mathcal{D}$ is simple normal crossing,
therefore the strata $Z$ defined in~\ref{strata} is indeed a simple normal crossing strata.

Consider $Y$ a smooth projective variety of dimension $n-k$
and $N$ a free finitely generated abelian group of rank $k$. 
Let $\mathcal{D}$ be a simple normal crossing pp-divisor on $(Y,N)$ with $\sigma=\sigma(\mathcal{D})$ its tail cone.
We write $\mathcal{D}= \sum_D \Delta_D \otimes D$,
and denote by $Z:=Z_{D_1,\dots,D_r}$ the strata of the prime divisors $D_1,\dots, D_r$.
Then, we can describe a formal neighborhood of the fiber $\pi^{-1}(Z)$
as follows: 
Consider the finitely generated free abelian group $N' =\zz^r \times  N $ and the cone
\[
\sigma(\mathcal{D},Z):= \langle (0,\sigma), (e_1, \Delta_{D_1}), \dots, (e_r, \Delta_{D_r}) \rangle \subset N'_\qq
\]
where the $e_i$'s give the canonical basis of $\zz^r$.
Thus, the formal neighborhood of a closed 
point of $\pi^{-1}(Z)$ is isomorphic to that 
of a corresponding closed point of
\[ 
X(\sigma(\mathcal{D},Z))\times Z.
\]
Therefore, given a divisorial fan $\mathcal{S}$ on $(Y,N)$
and $Z\subsetneq Y$ a strata of the divisorial fan,
the formal neighborhood of a closed point of $\pi^{-1}(Z)$ 
for the good quotient $\pi \colon \widetilde{X}(\mathcal{S})\rightarrow Y$
is isomorphic to the formal neighborhood of a 
corresponding closed point of
\[
X(\Sigma(\mathcal{S},Z))\times Z.
\] 
where $\Sigma(\mathcal{S},Z)$ is the fan in $N'$
given by the cones $\sigma(\mathcal{D},Z)$ for all $\mathcal{D}\in \mathcal{S}$.
Hence, given an analytic tubular 
neighbourhood $W_Z$ of $Z$ (in the 
usual sense of manifolds~\cite{bt}*{pag. 66})
the preimage $\pi^{-1}(W_Z)$ admits a natural structure of topological fibration
with base $Z$ and fiber $X(\Sigma(\mathcal{S},Z))$, given by the composition of the good quotient
$\pi^{-1}(W_Z)\rightarrow W_Z$ and the retraction $W_Z\rightarrow Z$.
\end{remark}

\begin{remark}
From the above description, we can see 
that the fiber over a closed point $y\in Z$ 
is isomorphic to the toric bouquet associated to the polyhedron
\[
\mathcal{D}_y = \sum_{i} \Delta_{D_i} \subset N_\qq.
\]
For the definition of toric bouquet, see~\cite[Section 2.2]{AIPSV}.
\end{remark}

\begin{notation}\label{lattice}
Let $\sigma$ be a rational polyhedral cone in $N_\qq$.
Denote by $N_\sigma$ the subgroup of $N$ 
generated by $\sigma \cap N$
and by $N(\sigma)$ the lattice quotient $N/N_\sigma$.
By ~\cite[Theorem 12.1.10]{CLS}, we know that 
\[
\pi_1(X(\sigma)) \simeq N(\sigma).
\]
Observe that $N(\sigma)$ is a free finitely generated abelian group.
Given a fan $\Sigma$ of polyhedral cones in $N_\qq$,
we denote by $N_\Sigma$ the semigroup generated by 
$\langle N_\sigma \mid \sigma \in \Sigma \rangle$,
and by $N(\Sigma)$ the quotient $N/N_\Sigma$.
By ~\cite[Theorem 12.1.10]{CLS}, we know that 
\begin{equation}\label{toricfund}
\pi_1(X(\Sigma))\simeq N(\Sigma).
\end{equation}
Observe that $N(\Sigma)$ is a finitely generated abelian group,
however it may be not free.
Given a basis $\{t_1,\dots, t_k\}$ of the lattice $N$,
a presentation for the fundamental group is the following
\begin{equation}
\label{pi1-toric}
 \pi_1(X(\Sigma)) \simeq \langle t_1, \dots, t_k \mid \mathcal{R}(\Sigma)) \rangle,
\end{equation}
where $\mathcal{R}(\Sigma)$ is the set of monomials $t_1^{n_1}\cdots t_k^{n_k}$,
where $(n_1,\dots,n_k)$ runs over all the bases 
of the lattices $N_{\tau}\subseteq N$ for each $\tau\in \Sigma$.

Using the notation of Remark~\ref{ngbhd}, 
we can see that the vectors $(e_i,0)$ of $N'$
can be realized as loops on $Y$ which goes around the divisor $D_i$.
Indeed, on a formal neighborhood of the generic point $\eta$ of $Z$, the variety $Y$ is 
analytically diffeomorphic to $\mathbb{A}^{r}_\mathbb{C}$. 
We denote by $\{0\}_r$ the origin of the affine space $\mathbb{A}^{r}_{\mathbb{C}}$, 
and by $S^{2r-1}$ the sphere of elements of norm one in $\mathbb{A}^{r}_{\mathbb{C}}$.
\end{notation}

\begin{construction}\label{cons}
Consider $\mathcal{S}$ to be a divisorial fan on $(Y,N)$,
such that the $\mathbb{T}$-variety $X(\mathcal{S})$ has log terminal singularities. 
By Lemma~\ref{resolution}, we can consider a resolution of singularities 
$\psi \colon Y' \rightarrow Y$, such that the pull back divisorial fan
$\psi^*(\mathcal{S})$ is simple normal crossing.
Then, by Lemma~\ref{toroidal}, we know that the $\mathbb{T}_N$-variety 
$\widetilde{X}(\psi^*(\mathcal{S}))$ is toroidal.
Since $\widetilde{X}(\psi^*(\mathcal{S}))$ has toroidal singularities it is log terminal (see, e.g. ~\cite[Theorem 11.4.24]{CLS}).
Therefore by ~\cite[Theorem 1.1]{Tak}, we know that the birational contraction
$r \colon \widetilde{X}(\psi^*(\mathcal{S})) \rightarrow X(\mathcal{S})$ 
induces an isomorphism of fundamental groups
\[
r_* \colon \pi_1( \widetilde{X}(\psi^*(\mathcal{S})) ) \simeq \pi_1(X(\mathcal{S})).
\]
So we can assume, without loss of generality,
that the divisorial fan of $X = X(\mathcal S)$
is simple normal crossing and $\widetilde X = X$.
We now propose a procedure to describe the 
fundamental group of such a variety $X$.
Let $V\subseteq Y$ be the trivial open set 
of the divisorial fan $\mathcal{S}$,
so that we have an isomorphism
\[
 \pi^{-1}(V) \simeq V\times  X(\Sigma(\mathcal S)),
\]
where $X(\Sigma(\mathcal S))$ is the general 
fiber of $\pi\colon X\to Y$, defined by the divisorial 
fan $\Sigma(\mathcal S)$.
By ~\cite[Theorem 12.1.5]{CLS} the inclusion
$\pi^{-1}(V)\to X$ induces a surjection of
fundamental groups.
By ~\cite[Theorem 12.1.10]{CLS} the fundamental
group of the toric variety $X_0$ is isomorphic
to $N(\Sigma(\mathcal S))$. Thus the above surjection
becomes
\[
\pi_1(V) \times N(\Sigma(\mathcal S)) \rightarrow \pi_1( X).
\]
Our task now is to describe the kernel of the
above homomorphism.
Given a codimension $r := \codim(Z)$ 
strata $Z$ of the divisorial fan $\mathcal{S}$,
let $W_Z$ be a formal neighborhood of $Z$
and let $V_Z := W_Z\cup V$.
We may assume without loss of generality that
if $\overline{Z}\cap \overline{Z'} =\emptyset$
then $W_{Z}\cap W_{Z'}=\emptyset$,
and if $\overline{Z}\supset \overline{Z'}$
then $W_{Z}\supset W_{Z'}$.

\begin{lemma}\label{surj}
Let $U\subseteq Y$ be an open subset,
then $\pi_*\colon \pi_1(\pi^{-1}(U))
\to \pi_1(U)$ is surjective.
\end{lemma}
\begin{proof}
Let $V\subseteq Y$ as usual the open subset
over which $\pi\colon X\to Y$ is trivial.
We have a commutative diagram
\[
 \xymatrix{
  \pi_1(\pi^{-1}(U\cap V))\ar[r]\ar[d]^-{\pi_*} 
  & \pi_1(\pi^{-1}(U))\ar[d]^-{\pi_*}\\
  \pi_1(U\cap V)\ar[r]
  & \pi_1(U)
 }
\]
where the non-labelled arrows are induced by 
the inclusion, and thus are surjections by
~\cite[Theorem 12.1.5]{CLS}. By the triviality
of $\pi$ over $V$ we deduce that the left hand 
side pushforward is surjective (projection onto
the first factor), so that the second pushforward 
must be surjective as well.
\end{proof}

\begin{lemma}\label{iso}
If $N(\Sigma(\mathcal S,Z))$ is trivial and 
$U\subseteq W_Z$ is a tubular neighborhood 
of the non-empty intersection $U\cap Z$, then 
$\pi_*\colon \pi_1(\pi^{-1}(U))\to \pi_1(U)$
is an isomorphism.
\end{lemma}
\begin{proof}
Let $\rho\colon U\to U\cap Z$ be a retraction
whose fibers are isomorphic to an open toric subvariety $\mathbb A^r$.
By Remark~\ref{ngbhd}, the composition
$\pi\circ\rho\colon\pi^{-1}(U)\to U\cap Z$ is a fibration
with fiber an open toric subvariety of $X(\Sigma(\mathcal{S},Z))$ containing a fixed point. Passing to 
the long exact sequence of homotopy groups
and recalling the isomorphism 
$\pi_1(X(\Sigma(\mathcal S,Z))\simeq 
N(\Sigma(\mathcal{S},Z))$ we get the following
exact sequence
\[ \xymatrix@1{
  \cdots\ar[r]
  & N(\Sigma(\mathcal{S},Z))\ar[r]
  & \pi_1(\pi^{-1}(W_Z))\ar[r]^-{\pi_*\circ\rho_*} 
  & \pi_1(Z)\ar[r]
  & 1
 }
\]In particular if $N(\Sigma(\mathcal{S},Z))$ is trivial
then $\pi_*\circ\rho_*$ is an isomorphism, so that
$\pi_*\colon \pi_1(\pi^{-1}(U))\to \pi_1(U\cap Z)$
is injective. We conclude by Lemma~\ref{surj}.
\end{proof}

\begin{remark}\label{fibrations}
Let $\rho\colon W_Z\to Z$ be a retraction
whose fibers are isomorphic to $\mathbb A^r$.
By Remark~\ref{ngbhd}, the composition
$\pi\circ\rho\colon\pi^{-1}(W_Z)\to Z$ is a fibration
with fiber $X(\Sigma(\mathcal{S},Z))$.
Moreover the fibers of the restriction of $\rho$ to 
$W_Z\cap V$ are isomorphic to $\mathbb A^r-\{0\}_r$,
so that the fibers of the restriction of $\pi\circ\rho$ to 
$\pi^{-1}(W_Z\cap V)$ are isomorphic to 
$(\mathbb{A}^r_{\mathbb{C}} -\{0\}_r)
\times X(\Sigma(\mathcal{S}))$,
which are homotopic to $S^{2r-1} \times X(\Sigma(\mathcal{S}))$.
We have a commutative diagram
\[
 \xymatrix@C=20pt{
  \mathbb A^r\setminus\{0\}_r\ar[r] 
  & W_Z\cap V\ar[r]^-\rho 
  & Z\\
  (\mathbb A^r\setminus\{0\}_r)\times X(\Sigma(\mathcal S))\ar[r]\ar[u]^-\pi\ar[d]
  & \pi^{-1}(W_Z\cap V)\ar[r]^-{\pi\circ\rho}\ar[u]^-\pi\ar[d] 
  & Z\ar[u]\ar[d]\\
  X(\Sigma(\mathcal{S},Z))\ar[r]
  & \pi^{-1}(W_Z)\ar[r]^-{\pi\circ\rho} 
  & Z
 }
\]
where the non-labelled arrows are inclusions.
Passing to fundamental groups, recalling
that $S^{2r-1}$ is a strong deformation retract
of $\mathbb A^r\setminus\{0\}_r$, recalling
the isomorphism $\pi_1(X(\Sigma(\mathcal S))
\simeq N(\Sigma(\mathcal S))$ and the isomorphism
$\pi_1(X(\Sigma(\mathcal S,Z))\simeq 
N(\Sigma(\mathcal{S},Z))$
we obtain the following commutative diagram
\begin{equation}
\label{diagram}
\begin{gathered}
 \xymatrix@C=20pt{
  \pi_1(S^{2r-1})\ar[r] 
  & \pi_1(W_Z\cap V)\ar@{->>}[r]^-{\rho_*}
  & \pi_1(Z)\\
  \pi_1(S^{2r-1})\times N(\Sigma(\mathcal S))\ar[r]\ar[u]^-{\pi_*}\ar[d]_-{\alpha_Z}
  & \pi_1(\pi^{-1}(W_Z\cap V))\ar@{->>}[r]^-{\pi_*\circ\rho_*}\ar[u]^-{\pi_*}\ar[d]_-{\beta_Z}
  & \pi_1(Z)\ar[u]\ar[d]\\
  N(\Sigma(\mathcal{S},Z))\ar[r]
  & \pi_1(\pi^{-1}(W_Z))\ar@{->>}[r]^-{\pi_*\circ\rho_*} 
  & \pi_1(Z)
 }
\end{gathered}
\end{equation}
Moreover each row in the above diagram is
part of the long exact sequence of homotopy
groups induced by a fibration, so that the last
map of each row is a surjection.
If $r>1$ then $\alpha_Z$
is the surjection $N(\Sigma(\mathcal S))\rightarrow N(\Sigma(\mathcal{S},Z))$
induced by the inclusion of lattices $N_{\Sigma(\mathcal{S})} \hookrightarrow N_{\Sigma(\mathcal{S},Z)}$.
If $r=1$, the generator of $\pi_1(S^1)\simeq \zz$
is the loop corresponding to $(e_1,0)$ in the notation 
of Remark~\ref{ngbhd}.
Observe that in both cases $\alpha_Z$
is a homomorphism of abelian groups.
\end{remark}

\begin{remark}
\label{locglob}
The following commutative diagram of inclusions
\begin{equation}\nonumber
  \xymatrix@C=20pt{
  \pi^{-1}(W_Z \cap V)  \ar[d]\ar[r] & \pi^{-1}(W_Z) \ar[d]    \\
   \pi^{-1}(V) \ar[r] &  \pi^{-1}(V_Z) }
\end{equation}
induces a commutative pushout diagram of 
fundamental groups by the Seifert-van Kampen 
theorem. Using the triviality of $\pi$ over $V$ 
we have a pushout diagram
\begin{equation}\label{vankampen}
\begin{gathered}
  \xymatrix@C=20pt{
  \pi_1(W_Z \cap V)\times N(\Sigma(\mathcal S))  \ar[d]_-{\imath_*\times {\rm id}}\ar[r]^-{\beta_Z} &  \pi_1(\pi^{-1}(W_Z)) \ar[d]    \\
 \pi_1(V) \times N(\Sigma(\mathcal S)) \ar[r]_-{\gamma_Z} &  \pi_1( \pi^{-1}(V_Z)) }
\end{gathered}
\end{equation}
where $\imath\colon W_Z \cap V \to V$ 
is the inclusion.
\end{remark}
\end{construction}

\section{Applications}
\label{sec:app}

\subsection{Simply connectedness of the spectrum of the Cox ring}
\label{subsec:fixed}

The aim of this subsection is to use Construction~\ref{cons} to prove Theorem~\ref{Thm1} and Corollary~\ref{Cor1}.

\begin{remark}\label{ind}
Theorem~\ref{Thm1} is independent of the choice of the resolution of singularities of the Chow quotient of $X$.
Indeed, let $Y_1$ and $Y_2$ be two resolution of singularities of the Chow quotient of $X$.
Then, using resolution of singularities we can find a common resolution $Y'$ of $Y_1$ and $Y_2$, 
so by ~\cite[Theorem 1.1]{Tak} we deduce that $\pi_1(Y')\simeq \pi_1(Y_2)$ and $\pi_1(Y')\simeq \pi_1(Y_1)$, concluding the claim.
\end{remark}

\begin{remark}\label{lc}
We point out that the log terminal condition in Theorem~\ref{Thm1} cannot be weakened:
let $H=\mathcal{O}_{\mathbb{P}^2}(1)|_C$ be an ample divisor on a plane elliptic curve $C\subset \mathbb{P}^2$, and let
\[
X := \Spec\left( \oplus_{m\in \mathbb{Z}_{\geq 0}} H^0(Y, \mathcal{O}_Y(mH)) \right).
\]
Then $X$ is a $\mathbb{T}$-variety of complexity one with an isolated log canonical singularity at the vertex,
and $X$ is contractible. Therefore, we have that $\pi_1(X)$ is trivial, while its Chow quotient $C$ has
non-trivial fundamental group.
\end{remark}

\begin{proof}[Proof of Theorem~\ref{Thm1}]
We use the notation of Construction~\ref{cons}.
Since $X$ admits a good $\mathbb{T}$-action, 
there is a $\mathbb{T}$-equivariant isomorphism
$X(\mathcal{D})\simeq X$, where $\mathcal{D}$ 
is a pp-divisor on $(Y,N)$, the variety $Y$ is projective,
and $\sigma(\mathcal{D}) \subset N_\qq$ is a full-dimensional cone (see, e.g.,~\cite[Section 4]{LS}).
Let $\psi \colon Y' \rightarrow Y$ be a resolution of singularities of $Y$
and let $\psi^*(\mathcal{D})$ be a pull back of the pp-divisor to $Y'$.
By Remark~\ref{ind}, it suffices to prove that the good quotient 
$\pi \colon \widetilde{X}(\psi^*(\mathcal{D})) \rightarrow Y'$ 
induces an isomorphism of the fundamental groups.
Also by ~\cite[Theorem 1.1]{Tak} 
the birational contraction $r\colon
\widetilde{X}(\psi^*(\mathcal{D}))\to X(\mathcal{D})$ 
induces an isomorphism of fundamental groups.
Thus from now on, without loss of generality, we 
will assume
\[
 X = \widetilde{X}(\psi^*(\mathcal{D}))
\]
and $Y$ smooth.
Observe that since $\sigma(\mathcal{D})$ is full-dimensional, 
the group $N(\sigma(\mathcal{D}))$ is trivial.
In particular, the group $N(\sigma(\mathcal{D},Z))$ 
is trivial for each strata $Z$ of $\mathcal{D}$, so
that
\[
 N(\Sigma(\mathcal S)) = N(\Sigma(\mathcal{S},Z)) = 0.
\]
Thus for every strata $Z$ of $\psi^*(\mathcal{D})$,
the map $\pi_*\colon\pi_1(\pi^{-1}W_Z))\to \pi_1(W_Z)$
is an isomorphism by Lemma~\ref{surj}.
By Remark~\ref{locglob} and the unicity
of pushout, the pushforward
\begin{equation}
\label{pi-isom}
 \pi_*\colon\pi_1(\pi^{-1}(V_Z))\to \pi_1(V_Z)
\end{equation}
is an isomorphism. 
Now, let 
\[
V_{Z_1,\dots, Z_k}:= V_{Z_1} \cup \dots \cup V_{Z_k}
\qquad
W_{Z_1,\dots,Z_k}:= W_{Z_1}\cup \dots \cup W_{Z_k}
\]
be the union of the open sets corresponding to 
$k$ different strata. 
The following is a pushout diagram by the
Seifert-van Kampen theorem.
\[
 \xymatrix@C=20pt{
  \pi_1( \pi^{-1}( W_Z \cap W_{Z_1,\dots,Z_k} \cap V)) \ar[d] \ar[r] 
  & \pi_1(\pi^{-1}(V)) \ar[d]    \\
   \pi_{1}(\pi^{-1}(W_Z\cap W_{Z_1,\dots,Z_k})) \ar[r] &  \pi_{1}(\pi^{-1}(V_Z\cap V_{Z_1,\dots,Z_k}))
   }
\]
If we apply $\pi_*$ we get another pushout 
diagram by the same theorem. Moreover 
$\pi_*$ is an isomorphism on the left hand
side groups by Lemma~\ref{iso}, and is an 
isomorphism on the top-right group by the
triviality of $N(\Sigma(\mathcal S))$ and the fact
that $\pi^{-1}(V)\simeq V\times 
X(\Sigma(\mathcal S))$. By the unicity
of pushouts we deduce that 
\begin{equation}\label{intersection}
 \pi_*\colon \pi_1( \pi^{-1}( V_Z \cap V_{Z_1,\dots,Z_k})) \rightarrow 
 \pi_1( V_Z\cap V_{Z_1,\dots, Z_k})
\end{equation}
is an isomorphism as well.
Now, we prove by induction on $k$ that 
\[
\pi_*\colon \pi_1 ( \pi^{-1}( V_{Z_1,\dots,Z_k})) \rightarrow \pi_1(V_{Z_1,\dots,Z_k}).
\]
is an isomorphism. The case $k=1$ is
~\eqref{pi-isom}.
If $k\geq 2$ we have a pushout diagram 
\begin{equation}\nonumber  \xymatrix@C=20pt{
  \pi_1(\pi^{-1}(V_{Z_1,\dots,Z_{k-1}}\cap V_{Z_k}))  \ar[d] \ar[r] & \pi_1(\pi^{-1}(V_{Z_k})) \ar[d]    \\
   \pi_{1}(\pi^{-1}(V_{Z_1,\dots, Z_{k-1}})) \ar[r] &  \pi_{1}(\pi^{-1}(V_{Z_1,\dots,Z_k})) }
\end{equation}
induced by the inclusions. By~\eqref{pi-isom},
the inductive hypothesis, the isomorphism~\eqref{intersection}, 
and the uniqueness of pushouts up to isomorphism, 
we conclude the claim.
Since $X$ can be covered by finitely many strata, 
the above argument shows that 
$\pi_* \colon \pi_1(X)\rightarrow \pi_1(Y)$
is an isomorphism, proving the statement.
\end{proof}

\begin{proof}[Proof of Corollary~\ref{Cor1}]
By ~\cite{Tak2,Xu}, and ~\cite[Corollary 1.3.2]{BCHM} 
it is known that Fano varieties are simply connected Mori dream spaces.
Moreover, the singularities of the spectrum $\overline{X}$ of the Cox ring of a Fano variety $X$ are log terminal ~\cite{GOST},
and $\overline{X}$ has a good action for the Picard torus ~\cite{AW}. 
Thus, we can apply Theorem~\ref{Thm1}, to deduce that $\pi_1(\overline{X})\simeq \pi_1(X) \simeq \{0\}$.
\end{proof}

\subsection{Rational log terminal $\mathbb{T}$-varieties of complexity one}
\label{subsec:comp1}

The aim of this section is to give a presentation
of the possible fundamental groups
of rational log terminal $\mathbb{T}$-varieties of complexity one. 

\begin{notation}\label{notation3.3}
Consider a divisorial fan $\mathcal{S}$ on $(Y,N)$.
For each $\mathcal{D}\in \mathcal{S}$ and $p\in Y$
denote by $\mathcal Q(\mathcal D,p)$ a basis
of the lattice $N_{\sigma(\mathcal{D},p)}\subseteq
\zz \times N$ and by 
\[
 \mathcal{B}(\mathcal{D},p) 
 := \left\{\left( \frac{v_2}{v_1},\dots, \frac{v_{k+1}}{v_1} \right)\in N_\qq
 \mid (v_1,\dots, v_{k+1})\in \mathcal Q(\mathcal D,p)\right\}.
\]
Given a point $v\in N_\qq$, we will denote by $\mu(v)$ the smallest positive integer
such that $\mu(v)v\in N$. 
Observe that $\mu(v)\leq v_1$ for $v\in \mathcal{B}(\mathcal{D},p)$.
\end{notation}

\begin{theorem}\label{comp1fund}
Let $\mathcal{S}$ be a divisorial fan on $(\pp^1,N)$,
assume that $X(\mathcal{S})$ has log terminal singularities
and let $\{p_1,\dots, p_{r} \} \subsetneq Y$ be the complement 
of the trivial locus of $\mathcal{S}$.
Then $\pi_1(X(\mathcal{S}))$ admits a presentation
with generators
\[
b_1,\dots, b_{r}, t_1,\dots, t_k, 
\]
where $k$ is the rank of the acting torus,
and relations
\begin{align*}
& b_1\cdots b_r\\
& [t_i,t_j] & \text{ for any $i,j\in\{1,\dots,k\}$}\\
& [t_i,b_j] & \text{ for } i\in \{1,\dots, k\} \text{ and } j\in \{1\dots, r\}, \\
& \mathcal{R}( \Sigma(\mathcal{S}) ) & \text{ where $\Sigma(\mathcal{S})$ is the tail fan of $\mathcal{S}$,}  \\
& t^{\mu(v)v} b_j^{\mu(v)}& \text{ for each $j\in \{1,\dots, r\}, v\in \mathcal{B}(\mathcal{D},p_j)$ and $\mathcal{D}\in \mathcal{S}$.} 
\end{align*}
\end{theorem}

\begin{proof}
We use the notation of Construction~\ref{cons}.
Let 
\[
\mathcal{G}(V)  := \{ b_1,\dots, b_{r}, t_1,\dots, t_k \}
\]
and let 
\[
\mathcal{R}(V) := \mathcal{R}(\Sigma(\mathcal S)) \cup
\{b_1\dots b_r\} \cup\{[t_i,t_j] \text{ for any $i,j$}\}\cup\{[t_i,b_j] \text{ for any $i,j$}\}.
\]
By the triviality of $\pi$ over $V$, the
formula~\eqref{pi1-toric} and the fact 
that $V = \mathbb P^1\setminus
\{p_1,\dots,p_r\}$, we have that
\[
 \pi_1(\pi^{-1}(V)) \simeq  \langle \mathcal{G}(V) \mid \mathcal{R}(V) \rangle.
\]
Observe that the stratification of $\pp^1$ induced by $\mathcal{S}$ is given by the sets $V$ and $p_1,\dots, p_{r}$.
Moreover, for any $j\in \{1,\dots, r\}$,
by~\eqref{diagram}, with $r=1$ and
$Z = p_j$, there is a commutative diagram
\[
 \xymatrix@C=20pt{
  \mathbb Z\times N(\Sigma(\mathcal S))\ar[r]^-\simeq\ar[d]_-{\alpha_j}
  & \pi_1(\pi^{-1}(W_{p_j}\cap V))\ar[d]_-{\beta_j}\\
  N(\Sigma(\mathcal{S},p_j))\ar[r]^-\simeq
  & \pi_1(\pi^{-1}(W_{p_j}))
 }
\]
whose horizontal arrows are isomorphisms
and the vertical arrows are surjections.
Recall that $N(\Sigma) := N/N_\Sigma$,
as defined in Notation~\ref{lattice}.
The homomorphism $\alpha_j$
is induced by the toric embedding
\[
 \cc^* \times X(\Sigma(\mathcal{S})) \rightarrow X(\Sigma(\mathcal{S},p_j)), 
\]
which is induced by the fan inclusion $0\times \Sigma(\mathcal{S}) \hookrightarrow \Sigma(\mathcal{S},p_j)$.
Thus, by~\cite[Theorem 12.1.10]{CLS},
$\alpha_j$ is induced by the inclusion of lattices 
$0\times N_{\Sigma(\mathcal{S})} \hookrightarrow N_{\Sigma(\mathcal{S},p_j)}$
so that
\[
 \ker(\alpha_j) = \frac{N_{\Sigma(\mathcal S,p_j)}}{0\times N_{\Sigma(\mathcal S)}}.
\]
The lattice $N_{\Sigma}$ is generated by the 
integral points of $\Sigma$. Thus $\ker(\alpha_j)$
is generated by a basis of lattice points of 
$N_{\Sigma(\mathcal S,p_j)}$.
If we denote by $b_j$ a generator of the 
$\mathbb Z$ group in the domain of $\alpha_j$,
then the kernel of $\alpha_j$ is generated
by the set
\[
\mathcal{B}_j:=\{b_j^{\mu(v)}t^{\mu(v)v} \mid \text{for each $v\in \mathcal{B}(\mathcal{D},p_j)$ and $\mathcal{D}\in \mathcal{S}$}\}
\]
because, by Notation~\ref{notation3.3},
the map $v\mapsto (\mu(v),\mu(v)v)$ 
is a bijection between $\mathcal{B}(\mathcal{D},p_j)$ 
and a basis of $N_{\sigma(\mathcal{D},p_j)}$.
Thus we have a presentation
\[
\pi_1( \pi^{-1}(W_{p_j})) \simeq \langle b_j, t_1,\dots, t_k \mid \mathcal{R}(\Sigma(\mathcal{S})), \mathcal{B}_j \rangle.
\]
Observe that we have a pushout diagram 
\begin{equation}\nonumber
  \xymatrix@C=20pt{
  \mathbb Z \times N(\Sigma(\mathcal S))  \ar[d]^{i\times {\rm id}} \ar[r]^-{\alpha_j} 
  &  N(\Sigma(\mathcal{S},p_j)) \ar[d]    \\
  \pi_1(V) \times N(\Sigma(\mathcal S)) \ar[r]^{\gamma_j} &  \pi_1( \pi^{-1}(V_{p_j})) }
\end{equation}
where the homomorphism $i$ is the inclusion 
$\langle b_j \rangle \hookrightarrow \langle b_1,\dots, b_{r}\rangle$.
Thus the kernel of $\gamma_j$ is the smallest normal subgroup 
of $\pi_1(V)\times N(\Sigma(\mathcal S))$ containing 
$(i\times {\rm id})(\ker(\alpha_j))$.
So, we obtain the following presentation 
\[
 \pi_1(\pi^{-1}(V_{p_j}))
 \simeq
 \langle 
  \mathcal{G}(V) \mid \mathcal{R}(V), \mathcal{B}_j 
 \rangle.
\] 
By repeatedly applying the Seifert-van Kampen 
Theorem we conclude that the fundamental group 
of $X$ is 
\[
 \pi_1(X) 
 \simeq 
 \langle \mathcal{G}(V) \mid \mathcal{R}(V), \mathcal{B}_1,\dots, \mathcal{B}_{r} \rangle, 
\]
proving the statement.
\end{proof}

\begin{proof}[Proof of Corollary~\ref{Cor2}]
In this case, the pp-divisor can be written as 
\[
\mathcal{D}= \sum_{i=1}^{r+1}\left\{ \frac{e_i}{m_i} \right\} \otimes p_i 
\in
{\rm Pol}(\{0\}) \otimes_{\mathbb{Z}_{\geq 0}} \CaDiv(\pp^1)_\qq,
\]
since the tail cone of $\mathcal{D}$ is $\{0\}$ in $N_\qq \simeq \qq$.
Hence, we conclude by Theorem~\ref{comp1fund}.
Indeed, observe that in this case we have 
a single pp-divisor $\mathcal{D}$ and for every 
$p_i$ we have that the lattice $N_{\sigma(\mathcal{D},p_i)}$
is the lattice of $\zz\times N$ generated by 
\[
\mathcal{Q}(\mathcal{D},p_i) = \{ (m_i, e_i)\},
\]
and therefore we have
\[
\mathcal{B}(\mathcal{D},p_i)= \left\{ \frac{e_i}{m_i}\right\}.
\]
\end{proof}

\subsection{Finiteness of local fundamental group}
\label{subsec:finite}

In this subsection we study the finiteness of the local fundamental group of a
log terminal $\mathbb{T}$-singularity with a good torus action.

\begin{definition}
We say that the $\mathbb{T}$-variety $X$ has a trivial GIT decomposition
if there exists exactly one GIT quotient of $X$ of the expected dimension,
and moreover this coincides with the Chow quotient of $X$.
\end{definition}

\begin{proof}[Proof of Theorem~\ref{Thm2}]
The local fundamental group at the vertex is a local computation, so we may assume that the $\mathbb{T}$-variety is affine.
By ~\cite[Theorem 1]{AH05}, we have a $\mathbb{T}$-equivariantly isomorphism 
$X\simeq X(\mathcal{D})$
for some pp-divisor $\mathcal{D}$ on $(Y,N)$,
where $Y$ is the Chow quotient of $X$.
Since we are assuming that the $\mathbb{T}$-action on $X(\mathcal{D})$ is good, 
we know that $Y$ is projective
and $\sigma=\sigma(\mathcal{D})\subset N_\qq$ is a full-dimensional cone 
(see, e.g.,~\cite[Section 4]{LS}).
Moreover, by ~\cite{LS} we know that the trivial GIT decomposition implies that 
$\mathcal{D}(u)$ is an ample $\qq$-divisor for every $u\in {\rm relint}(\sigma^\vee)$.
We will write
\[
\mathcal{D}= \sum_{D\subsetneq Y} \Delta_D \otimes D, 
\]
where the sum runs over a finite set of $\qq$-Cartier $\qq$-divisors on $Y$
and $\Delta_D$ are $\sigma$-polyhedra.
For each vertex $v\in \Delta_D$ we recall that
$\mu(v)$ is the smallest positive integer such that
$\mu(v)v$ is a lattice point of $N_\qq$, by 
\[
\mu_D := \max \{ \mu(v) \mid \text {$v$ is a vertex of $\Delta_D$}\},
\]
and by
\[
B := \sum_{D\subsetneq Y} \frac{\mu_D-1}{\mu_D} D,
\]
where the sum runs over the same $\qq$-Cartier $\qq$-divisors of $\mathcal{D}$.
By ~\cite{LS}*{Theorem 4.9}, we know that the pair $(Y,B)$ is a log Fano pair.
We claim that $Y$ has log terminal singularities. Indeed, the effective divisor $B$ is $\qq$-Cartier
and $K_Y+B$ is $\qq$-Cartier, so we conclude that $K_Y$ is $\qq$-Cartier, therefore $Y$ is $\qq$-Gorenstein.
Thus the klt property of $(Y,B)$ 
implies the log terminal property of $Y$ (see, e.g.,~\cite[Corollary 2.35]{KM98}).
Hence, by ~\cite[Theorem 1.1]{Tak} and ~\cite{Tak2} we conclude that 
$\pi_1(Y')\simeq \pi_1(Y)\simeq \{0\}$ for every resolution of singularities $\psi \colon Y'\rightarrow Y$.

Let $\psi \colon Y'\rightarrow Y$ be a 
resolution of singularities such that 
$\psi^*(\mathcal{D})$ is a pp-divisor 
with simple normal crossing support
and let $X(\mathcal S)\to \widetilde{X}(\psi^*(\mathcal{D}))$
be a toroidal resolution of singularities,
defined by a divisorial fan $\mathcal{S}$ on $Y'$.
Denote by $\phi\colon X(\mathcal{S})
\rightarrow X(\mathcal{D})$ the projective
birational morphism obtained by composition.
Therefore, we have a quotient
of smooth varieties $\pi' \colon X(\mathcal{S})\rightarrow Y'$.
By the assumption that a fiber of 
$\pi \colon X(\mathcal{D})\rightarrow Y$
over a codimension one point  
contains a smooth point,
we can take the resolution of singularities which does not blow-up such smooth points, 
so the analogous assumption holds for the morphism $\pi'$.
Moreover, the general fiber of $\pi'$ is isomorphic to 
$X(\Sigma)$, where $\Sigma$ is a simplicial 
refinement of $\sigma\subset N_\qq$.
Since $\sigma$ is full-dimensional, then
the rays of $\Sigma$ span $N_\mathbb Q$
so that, by ~\cite[Theorem 12.1.10]{CLS},
the fundamental group of 
$X(\Sigma)$ with the vertex
removed is a finite group $G(\Sigma)$.
Hence, we can apply~\cite[Lemma 1.5]{Nor}, 
to conclude that we have an exact sequence
\[
G(\Sigma) \rightarrow \pi_1( \phi^{-1}(X(\mathcal{D})-\{ x\})) \rightarrow \pi_1(Y')\rightarrow 1.
\]
Observe that by~\cite[Theorem 1.1]{Tak} we have that 
$\pi_1(Y')=\pi_1(Y)=1$, so we conclude that 
$\pi_1( \phi^{-1}(X(\mathcal{D})-\{ x\}))$ is
finite and thus $\pi_1( X(\mathcal{D})-\{x\})$ 
is finite as well, again by ~\cite[Theorem 1.1]{Tak}.
\end{proof}

\begin{corollary}
With the same assumptions than Theorem~\ref{Thm2}.
The inclusion of a general fiber of $\pi\colon X \dashrightarrow Y$ 
induces a surjection of the local fundamental group at the vertex of the general fiber 
onto the local fundamental group at the vertex $x\in X$.
\end{corollary}

\begin{proof}
Indeed, by the proof of the theorem 
we have a surjection
\[
 G(\Sigma) \rightarrow \pi_1( \phi^{-1}(X(\mathcal{D})-\{ x\}))
\]
and the latter group is isomorphic to 
$\pi_1(X(\mathcal{D})-\{x\})$ 
by~\cite[Theorem 1.1]{Tak}. 
\end{proof}

\section{Examples}
\label{sec:examples}

It is know that a germ of surface 
Du Val singularity $x\in X$,
is a log terminal $\mathbb{T}$-variety
of complexity one with trivial fundamental group.
Moreover, since Du Val singularities are quotients of $\cc^2$ by a binary polyhedral group $G\subsetneq {\rm SL}_2(\mathbb{C})$, 
the fundamental group of the punctured neighborhood $X\setminus \{x\}$ is isomorphic to $G$.
In this subsection, we recall an explicit computation of Du Val singularities as $\mathbb{T}$-varieties
and recover their fundamental group using 
Theorem~\ref{comp1fund}.

\begin{example}
By~\cite[Corollary 5.6]{LS}, we know that every quasi-homogeneous log-terminal surface singularity is isomorphic to the 
section ring of one of the following $\qq$-divisors on $\pp^1$
\[
\mathcal{D}= \left\{ \frac{e_1}{m_1} \right\} \otimes [0] + \left\{ \frac{e_2}{m_2} \right\} \otimes [1] + \left\{ \frac{e_3}{m_3} \right\} \otimes [\infty], \text{ and }
\frac{e_1}{m_1}+ \frac{e_2}{m_2}+ \frac{e_3}{m_3}>0.
\]
Here, the triple $(m_1,m_2,m_3)$ is one of the platonic triples $(1,p,q),(2,2,r),(2,3,3),\\(2,3,4)$ and $(2,3,5)$,
where $p,q\geq 1$ and $r\geq 2$. 
Moreover, the contraction morphism $r\colon \widetilde{X}(\mathcal{D})\rightarrow X(\mathcal{D})$ is
contracting a curve onto the vertex $x$ of $X(\mathcal{D})$. Therefore, the map
\[
\pi \colon \widetilde{X}(\mathcal{D}) - r^{-1}(x) \rightarrow \pp^1
\]
is a $\cc^*$-bundle with at most three non-reduced fibers. 
Then we can apply Corollary~\ref{Cor2} to give a presentation 
of the fundamental group of $\pi_1(X(\mathcal{D}))$ in terms 
of generators and relation as follows
\[
\langle b_1,b_2, t \mid [b_1,t],\, [b_2,t],\,  
t^{e_1}b_1^{m_1},\, t^{e_2}b_2^{m_2},\,  t^{e_3}b^{m_3}  \rangle
\]
where $b=(b_1b_2)^{-1}$.
\end{example}

\begin{example}\label{ex1}
By ~\cite[Theorem 5.7]{LS}, we know that every quasi-homogeneous canonical surface singularity 
is isomorphic to the section ring of one of the following $\qq$-divisors on $\pp^1$
\begin{align}
A_i \colon & \hspace{1cm} \left\{\frac{i+1}{i}\right\} \otimes [\infty] & i\geq 1. \nonumber \\
D_i \colon & \hspace{1cm} \left\{\frac{1}{2}\right\}\otimes [0] + \left\{\frac{1}{i-2}\right\}\otimes [1] + \left\{\frac{-1}{2}\right\}\otimes [\infty]  & i\geq 4. \nonumber \\
E_i \colon & \hspace{1cm} \left\{\frac{1}{3}\right\}\otimes [0] + \left\{\frac{1}{i-3}\right\}\otimes [1] + \left\{\frac{-1}{2}\right\}\otimes [\infty]  &i \in \{6,7,8\}. \nonumber
\end{align}
Remark that we correct some sign typos in the statement of
\cite[Theorem 5.7]{LS}. The proof therein remains valid without
corrections. Proceeding as in Example~\ref{ex1}, we can recover the
local fundamental groups of DuVal singularities
\begin{align*}
A_i \colon & \langle b \mid b^{i+1} \rangle & i\geq 1. \\
D_i \colon & \langle b_1, b_2 \mid b_1^2 = b_2^{i-2}= (b_1b_2)^2 \rangle & i\geq 4.\\
E_i \colon & \langle b_1, b_2 \mid b_1^3 = b_2^{i-3}= (b_1b_2)^2 \rangle & i\in \{6,7,8\}.
\end{align*}
More precisely, for the $E_8$ singularity we obtain the following presentation
\[
\langle b_1,b_2, t \mid [b_1,t],[b_2,t], tb_1^3, tb_2^5, t(b_1b_2)^2 \rangle 
\simeq 
\langle b_1, b_2 \mid b_1^3 = b_2^5 = (b_1b_2)^2 \rangle.
\]
The other computations are analogous.
\end{example}

\begin{example}
We give an example of a rational affine $\mathbb{T}$-variety of complexity one $X(\mathcal{D})$
with a good torus action and vertex $x\in X(\mathcal{D})$, such that the fundamental groups
$\pi_1(X(\mathcal{D}))$ and $\pi_1(X(\mathcal{D})-\{x\})$ are both trivial.
Consider the cone 
\[
\sigma= \langle (-1,1),(11,8)\rangle \subset N_\qq \simeq \qq^2
\]
and the proper polyhedral divisor given by
\[
\mathcal{D}= \Delta_0 \otimes [0] + \Delta_1 \otimes [1] + \Delta_{\infty}\otimes [\infty], 
\]
where
\begin{align}
&\Delta_0  = \sigma+ \left( \frac{2}{5}, \frac{1}{5}\right),\nonumber \\
&\Delta_1  = \sigma+ \left( \frac{1}{3}, \frac{1}{3}\right), \nonumber \\
&\Delta_\infty = \sigma+ {\rm conv}((0,0),(1,0)),\nonumber
\end{align}
where ${\rm conv}$ denotes the convex hull.
By Theorem~\ref{Thm1}, we conclude that $\pi_1(X(\mathcal{D}))$ is trivial.
Moreover, we can apply Theorem~\ref{comp1fund} to see that 
$\pi_1(X(\mathcal{D})-\{x\})$ is isomorphic to
\[
\langle t_1,t_2,b_1,b_2 \mid [t_1,b_1],\dots, [t_2,b_2],
t_1^{-1}t_2, t_1^{11}t_2^8, t_1^2t_2b_1^5, t_1t_2b_2^3, (b_1b_2)^{-1}, t_1(b_1b_2)^{-1}
\rangle.
\]
Indeed, we have that 
\begin{align}
& \mathcal{Q}(\mathcal{D},\{0\}) = \{ (0,-1,1),(0,11,8),(5,2,1)\}, 
\nonumber \\
\nonumber & \mathcal{Q}(\mathcal{D},\{1\})= \{ (0,-1,1),(0,11,8),(3,1,1)\},\text{ and }\\
\nonumber & \mathcal{Q}(\mathcal{D},\{\infty\}) = \{ (0,-1,1),(0,11,8),(1,0,0),(1,1,0)\}.
\end{align}
From the first and last two relations we obtain $t_1=t_2=e$, 
$b_1^5=b_2^3=b_1b_2=e$, so that $b_1=b_2=e$. 
\end{example}

\end{document}